\documentclass{amsart}

\usepackage{color}
\usepackage{amsmath, amsthm, amssymb}
\usepackage{amsfonts}
\usepackage[ansinew]{inputenc}
\usepackage[dvips]{epsfig}
\usepackage{graphicx}
\usepackage[english]{babel}
\usepackage{hyperref}
\theoremstyle{plain}
\newtheorem{thm}{Theorem}
\newtheorem{cor}[thm]{Corollary}
\newtheorem{lem}[thm]{Lemma}
\newtheorem{prop}[thm]{Proposition}
\newtheorem{rem}[thm]{Remark}

\title[On a factorization of Riemann's $\zeta$ function and its computation]{On a factorization of Riemann's $\zeta$ function with respect to a quadratic field and its computation}

\author{Xavier Ros-Oton}

\address{Universitat Polit\`ecnica de Catalunya, Departament de Matem\`{a}tica
Aplicada I, Diagonal 647, 08028 Barcelona, Spain}
\email{xavier.ros.oton@upc.edu}


\date{}
\begin{document}

\maketitle

%\begin{center}
%{Universitat Polit\`ecnica de Catalunya, Departament de Matem\`{a}tica  Aplicada I, Diagonal 647, 08028 Barcelona, Spain}
%\end{center}
%
%\vspace{4mm}

\begin{abstract} Let $K$ be a quadratic field, and let $\zeta_K$ its Dedekind zeta function. In this paper we introduce a factorization of $\zeta_K$ into two functions, $L_1$ and $L_2$, defined as partial Euler products of $\zeta_K$, which lead to a factorization of Riemann's $\zeta$ function into two functions, $p_1$ and $p_2$. We prove that these functions satisfy a functional equation which has a unique solution, and we give series of very fast convergence to them. Moreover, when $\Delta_K>0$ the general term of these series at even positive integers is calculated explicitly in terms of generalized Bernoulli numbers.
\end{abstract}

\vspace{3mm}
\section{Introduction}
\label{intro}

Let $K$ be a quadratic field and let $\chi$ be the Dirichlet
character attached to $K/\mathbb{Q}$. Its Dedekind's zeta
function can be written as
\[\zeta_K(s)=\zeta(s)L(s,\chi),\] where $\zeta$ is Riemann's zeta
function and $L$ is the $L$-function associated with $\chi$ (see, for example, \cite{2}).
Hence, an
alternative factorization, for $\mathfrak{Re}(s)>1$, is the one given by the
partial products
\[\zeta_K(s)=\prod_{p|d}(1-p^{-s})^{-1}L_1(s)L_2(s),\] where $d=|\Delta_K|$ is the absolute value of the discriminant of $K$, and
\[L_1(s)=\prod_{\chi(p)=1}(1-p^{-s})^{-2},\qquad
L_2(s)=\prod_{\chi(p)=-1}(1-p^{-2s})^{-1}.\] Note that $L_1$ and
$L_2$ are obtained as partial Euler products of $\zeta(s)^2$ and
$\zeta(2s)$ respectively, so they converge and are non-zero for
$\mathfrak{Re}(s)>1$ and $\mathfrak{Re}(s)>1/2$ respectively.

Define now
\begin{equation}\label{p}p_1(s)=\prod_{\chi(p)=1}(1-p^{-s})^{-1}\qquad
\textrm{and}\qquad p_2(s)=\prod_{\chi(p)=-1}(1-p^{-s})^{-1}.\end{equation}
Then, we have that
\[L_1(s)=p_1(s)^2,\qquad L_2(s)=p_2(2s),\]
and thus it is equivalent
to study $L_1$ and $L_2$ or $p_1$ and $p_2$.
Note that  \[\zeta(s)=\prod_{p|d}(1-p^{-s})^{-1}p_1(s)p_2(s),\] and hence, $p_1$ and $p_2$ give a factorization of Riemann's zeta function.

The plan of the paper is as follows. In section 2 we see that $p_1$ and $p_2$ satisfy a functional equation. More precisely, we prove

\begin{thm}\label{th} The functions $p_1$ and  $p_2$ satisfy the functional equations
\begin{equation}\label{functeq}\frac{p_i(2s)}{p_i(s)^2}=q_i(s),
\qquad \lim_{\mathfrak{Re}(s)\rightarrow +\infty}p_i(s)=1,\qquad \mbox{for}\qquad i=1,2,\end{equation} where \begin{equation}\label{qs}q_1(s)=\frac{\zeta(2s)}{\zeta(s)L(s,\chi)}\prod_{p|d}(1+p^{-s}),\qquad q_2(s)=\frac{L(s,\chi)}{\zeta(s)}\prod_{p|d}(1-p^{-s})^{-1}.\end{equation}

Furthermore, these functional equations have a unique solution, so they completely determine the functions $p_1$ and $p_2$.
\end{thm}

Moreover, we shall see that the logarithm of the solution of this functional equation can be written as a series
\begin{equation}\label{series}\log p_i(s)=-\sum_{n=0}^{+\infty}\frac{\log q_i(2^ns)}{2^{n+1}},\qquad i=1,2,\end{equation} and hence, we will have an alternative expression of $p_1$ and $p_2$.

In section 3 we will see that the series given by \eqref{series} are of very fast convergence. We shall prove

\begin{thm}\label{prop2} Let $s$ be complex number such that $\mathfrak{Re}(s)\geq1$. Then,
\[p_1(2s)=\exp\left\{-\sum_{k=1}^{n}\frac{1}{2^k}\log q_1(2^ks)\right\}+o\left(2^{-2^n}\right),\]and \[p_2(2s)=\exp\left\{-\sum_{k=1}^{n}\frac{1}{2^k}\log q_2(2^ks)\right\}+o\left(2^{-2^n}\right).\]
\end{thm}

As a consequence, we will have a way to evaluate $p_1$ and $p_2$ at even positive integers when $\Delta_K$ is positive. This will be done by calculating explicitly the general term of the series in this case.

\vspace{5mm}

\section{The functional equation of $p_1$ and $p_2$}

First we prove that the functional equation appearing in Theorem \ref{th} has a unique solution and that this solution can be written as an infinite series. The statement of the result is the following.

\begin{prop}\label{prop1} Let $\Omega=\{s\in\mathbb{C} | \mathfrak{Re}(s)>1\}$, and $q$ an holomorphic function defined in $\Omega$, with $q(s)\neq0$ for all $s\in\Omega$ and
$\lim_{\mathfrak{Re}(s)\rightarrow +\infty}q(s)=1$. Then, the functional
equation \[\frac{p(2s)}{p(s)^2}=q(s),\qquad \lim_{\mathfrak{Re}(s)\rightarrow
+\infty}p(s)=1\] has a unique solution $p(s)$. In addition, the
solution can be written as
\[p(s)=\exp\left\{-\sum_{n\geq0}\frac{\log q(2^ns)}{2^{n+1}}\right\},\] and this series is absolutely convergent for all $s$ in $\Omega$.
\end{prop}

\begin{proof} Suppose that $p(s)$ satisfies the functional equation. Then, $p(s)\neq0$ for all $s\in\Omega$. This is because $p(s)=0$ implies $p(2s)=0$ and $p(2^ks)=0$ for $k=1,2,...$, which contradicts the hypothesis $\lim_{\mathfrak{Re}(s)\rightarrow
+\infty}p(s)=1$. Thus, we can define
\[f(s)=\frac{\log p(s)}{s},\qquad g(s)=\frac{\log q(s)}{2s},\] where $\log$ is the principal branch of the complex logarithm.
Taking logarithms to our functional equation and dividing by $2s$, we
have that
\[f(2s)=f(s)+g(s),\qquad \lim_{\mathfrak{Re}(s)\rightarrow +\infty}f(s)=0.\]
Writing this last equation for $s,2s,4s,8s,...,2^Ns$, and adding
them, we obtain that
\[f(2^{N+1}s)=f(s)+\sum_{n=0}^{N}g(2^ns).\] Since
$\mathfrak{Re}(s)>1$, then $\mathfrak{Re}(2^{N+1}s)\rightarrow+\infty$ when
$N\rightarrow\infty$, so
\[f(s)+\sum_{n=0}^{\infty}g(2^ns)=\lim_{N\rightarrow\infty}f(2^{N+1}s)=0,\] and \[\log p(s)=-\sum_{n\geq0}\frac{\log q(2^ns)}{2^{n+1}}.\]

Since \[\lim_{\mathfrak{Re}(s)\rightarrow +\infty}\log q(s)=0,\] the
sequence $\{\log q(2^ns)\}_{n\in\mathbb{N}}$ converges (it tends
to 0), and in particular it is bounded. Hence, there exists $M>0$
such that $|\log q(2^ns)|<M$, and then
\[\sum_{n\geq0}\left|\frac{\log q(2^ns)}{2^{n+1}}\right|\leq
\sum_{n\geq0}\frac{M}{2^{n+1}}=M,\] so the series is absolutely
convergent for all $s\in\Omega$.

Let us see that this function satisfies the functional
equation. We have that
\begin{eqnarray*}\log p(2s)-2\log p(s)&=&-\sum_{n\geq0}\frac{\log q(2^{n+1}s)}{2^{n+1}}+2\sum_{n\geq0}\frac{\log q(2^ns)}{2^{n+1}}\\
&=&-\sum_{n\geq1}\frac{\log q(2^ns)}{2^n}+\sum_{n\geq0}\frac{\log q(2^ns)}{2^n}\\
&=&\log q(s),\end{eqnarray*} and then,
\[\frac{p(2s)}{p(s)^2}=q(s).\] We now have to see that
$\lim_{\mathfrak{Re}(s)\rightarrow +\infty}p(s)=1$, or equivalently,
\[\lim_{\mathfrak{Re}(s)\rightarrow +\infty}\log p(s)=0.\] For it, fix $\varepsilon>0$.
Since $\lim_{\mathfrak{Re}(s)\rightarrow +\infty}q(s)=1$, then
$\lim_{\mathfrak{Re}(s)\rightarrow +\infty}\log q(s)=0$, and exists
$\sigma>0$ such that \[|\log q(s)|<\epsilon\ \ \ \mbox{for all}\ \ \ s\ \ \textrm{
with }\ \ \mathfrak{Re}(s)\geq\sigma.\]
Hence, if $\mathfrak{Re}(s)\geq \sigma$, then
\[|\log p(s)|\leq \sum_{n\geq0}\left|\frac{\log q(2^ns)}{2^{n+1}}\right|\leq \sum_{n\geq0}\frac{\varepsilon}{2^{n+1}}=\varepsilon,\]
and $\lim_{\mathfrak{Re}(s)\rightarrow +\infty}\log p(s)=0$, as claimed.

Note that, in fact, the branch of the logarithm is irrelevant, since when
we take exponentials, we will have
\[p(s)=\exp\left\{-\sum_{n\geq0}\frac{\log
q(2^ns)}{2^{n+1}}\right\},\] independently of the chosen branch.
\end{proof}

We can now give the:

\vspace{3mm}

\noindent \emph{Proof of Theorem \ref{th}.} On the one hand, it is clear that
$\lim_{\mathfrak{Re}(s)\rightarrow +\infty}p_i(s)=1$, $i=1,2$.

On the other hand, we have that
\begin{eqnarray*}p_1(s)\frac{p_2(2s)}{p_2(s)}&=&\prod_{\chi(p)=1}(1-p^{-s})^{-1}\frac{\prod_{\chi(p)=-1}(1-p^{-2s})^{-1}}{\prod_{\chi(p)=-1}(1-p^{-s})^{-1}}\\
&=&
\prod_{\chi(p)=1}(1-p^{-s})^{-1}\prod_{\chi(p)=-1}\left(\frac{1-p^{-2s}}{1-p^{-s}}\right)^{-1}\\
&=&\prod_{\chi(p)=1}(1-p^{-s})^{-1}\prod_{\chi(p)=-1}(1+p^{-s})^{-1}\\
&=&L(s,\chi),
\end{eqnarray*}
and since
\[p_1(s)=\frac{1}{p_2(s)}\zeta(s)\prod_{p|d}(1-p^{-s}),\]
then
\[\frac{p_2(2s)}{p_2(s)^2}=\frac{L(s,\chi)}{\zeta(s)}\prod_{p|d}(1-p^{-s})^{-1}.\]

Using now that
\[p_2(s)=\frac{1}{p_1(s)}\zeta(s)\prod_{p|d}(1-p^{-s}),\] we obtain
\[\frac{p_1(2s)}{p_1(s)^2}=\frac{p_2(s)^2}{p_2(2s)}\cdot\frac{\zeta(2s)\prod_{p|d}(1-p^{-2s})}{\zeta(s)^2\prod_{p|d}(1-p^{-s})^2}=\frac{\zeta(2s)}{\zeta(s)L(s,\chi)}\prod_{p|d}(1+p^{-s}).\]

The fact that these functional equations have an unique solution
follows from Proposition \ref{prop1}.
\qed

As a consequence of Proposition \ref{prop1} and Theorem \ref{th}, we obtain the following expression for $p_1(s)$ and $p_2(s)$.

\begin{cor} \label{corr} Let $p_1$ and $p_2$ be given by \eqref{p}. Then,
\[p_i(s)=\exp\left\{-\frac{1}{2}\sum_{n\geq0}\frac{\log q_i(2^ns)}{2^n}\right\}\qquad \mbox{for}\qquad i=1,2,\] where
\[q_1(s)=\frac{\zeta(2s)}{\zeta(s)L(s,\chi)}\prod_{p|d}(1+p^{-s}),\ \ \mbox{and}\ \ q_2(s)=\frac{L(s,\chi)}{\zeta(s)}\prod_{p|d}(1-p^{-s})^{-1}.\]
\end{cor}

These expressions will be used in the next section.

\vspace{5mm}

\section{Evaluating $p_1$ and $p_2$}

In this section we will calculate the order of convergence of the series given by Corollary \ref{corr}. We will see that this convergence is of order $2^{-2^n}$, i.e.,
\[p_i(2s)=\exp\left\{-\sum_{k=1}^{n}\frac{1}{2^k}\log q_i(2^ks)\right\}+o\left(2^{-2^n}\right),\]
and therefore this will be a better way to evaluate the functions $p_1$ and $p_2$ than the one given by the infinite products
\[p_1(s)=\prod_{\chi(p)=1}(1-p^{-s})^{-1}\qquad
\textrm{and}\qquad p_2(s)=\prod_{\chi(p)=-1}(1-p^{-s})^{-1}.\]

Moreover, we will provide the general term
of these series at even positive integers in the case $\Delta_K>0$. For it, we will use generalized Bernoulli numbers.

\begin{rem} Recall that \[f(n)=o\left(g(n)\right)\textrm{ means that }\lim_{n\rightarrow+\infty} \frac{f(n)}{g(n)}=0,\] and \[a(n)=b(n)+o\left(g(n)\right)\textrm{ means that }a(n)-b(n)=o\left(g(n)\right).\]\end{rem}

In order to prove Theorem \ref{prop2}, we will need two lemmata.

\begin{lem}\label{zeta} Let $\sigma$ be a real number, $\sigma>1$. Then, \[\frac{2^\sigma-1}{2^\sigma-2}<\zeta(\sigma)<\frac{2^\sigma}{2^\sigma-2}.\]
\end{lem}

\begin{proof} We make a partition of $\mathbb{N}$ in the sets $A_k=\{n\in\mathbb{N}: 2^k\leq n<2^{k+1}\}$, $k\geq1$. It is clear that $|A_k|=2^k$, and
that if $n\in A_k$, then $n^{-\sigma}\leq 2^{-k\sigma}$. Hence,
\begin{eqnarray*}\zeta(\sigma)&=&\sum_{n\in\mathbb{N}}n^{-\sigma}=\sum_{k\geq0}\sum_{n\in A_k}n^{-\sigma}< \sum_{k\geq0}\sum_{n\in
A_k}2^{-k\sigma}\\&=&\sum_{k\geq0}|A_k|{\cdot}2^{-k\sigma}=\sum_{k\geq0}2^k{\cdot}2^{-k\sigma}=\sum_{k\geq0}(2^{1-\sigma})^k\\&=&\frac{1}{1-2^{1-\sigma}}=\frac{2^\sigma}{2^\sigma-2}.\end{eqnarray*}

Using that if $n\in A_k$ then $n^{-\sigma}\leq 2^{-(k+1)\sigma}$, we obtain the other side of the inequality.
\end{proof}

\begin{lem}\label{fitaq} Let $s=\sigma+it$, with $\sigma\geq2$, and let $q_1$ and $q_2$ be given by \eqref{qs}. Then,
\[|\log q_i(s)|\leq \frac{16}{2^\sigma-2}\qquad\mbox{for}\qquad i=1,2,\] where $\log$ denotes the principal branch of the complex logarithm.
\end{lem}

\begin{proof} First we claim that
\begin{equation}\label{log}|\log(1+z)|\leq -\log(1-|z|),\end{equation}for each $|z|<1$.
To see it, it suffices to compare its power series:
\[|\log(1+z)|=\left|z-\frac{z^2}{2}+\cdots\right|\leq|z|+\frac{|z|^2}{2}+\cdots=-\log(1-|z|).\]

Now, using \eqref{log} and that
\[\left|\frac{1-p^{-s}}{1+p^{-s}}-1\right|=\frac{2p^{-\sigma}}{1-p^{-\sigma}},\] we get
\begin{eqnarray*} |\log q_i(s)| &=& \left|\log\prod_{\chi(p)=\pm1}\left(\frac{1-p^{-s}}{1+p^{-s}}\right)\right|\\
&\leq&  \sum_{\chi(p)=\pm1}\left|\log\left(\frac{1-p^{-s}}{1+p^{-s}}\right)\right|\\
&\leq&  \sum_{\chi(p)=\pm1}-\log\left(1-\frac{2p^{-\sigma}}{1-p^{-\sigma}}\right)\\
&=&\sum_{\chi(p)=\pm1} \log\left(\frac{1-p^{-\sigma}}{1-3p^{-\sigma}}\right).
\end{eqnarray*}
Moreover, since $\log(1+x)\leq x$ for each $x>0$, then
\[|\log q_i(s)| \leq \sum_{\chi(p)=\pm1}\left(\frac{1-p^{-\sigma}}{1-3p^{-\sigma}}-1\right)=\sum_{\chi(p)=\pm1}\frac{2}{p^{\sigma}-3}.\]
But since $\sigma\geq2$ then \[p^\sigma-3\geq \frac14 p^\sigma\] for each $p\geq2$, and therefore
\[|\log q_i(s)|\leq 8\sum_{\chi(p)=\pm1}p^{-\sigma},\qquad i=1,2.\]
Finally, by Lemma \ref{zeta} we have that
\[|\log q_i(s)|\leq 8\sum_{n\geq2}n^{-\sigma}\leq \frac{16}{2^\sigma-2},\qquad i=1,2,\]
and we are done.
\end{proof}

%\begin{prop} Let $r$ be a positive integer, and let $x_n$ and $y_n$ be the general term of the series which give $\log p_1(2r)$ and $\log p_2(2r)$, i.e. \[x_n=\frac{1}{2^n}\log q_1(2^nr),\qquad y_n=\frac{1}{2^n}\log q_2(2^nr).\] Then, \[x_n,\ y_n= o\left(2^{-2^n}\right).\] \end{prop}

By using the last Lemma, we will be able to bound the general term of the series which give $p_1$ and $p_2$, and from this, we will deduce Theorem \ref{prop2}.

\vspace{3mm}

\noindent \emph{Proof of Theorem \ref{prop2}.} Let $x_n$ and $y_n$ be the general term of the series which give $\log p_1(2s)$ and $\log p_2(2s)$, i.e. \[x_n=\frac{1}{2^{n+1}}\log q_1(2^ns),\qquad y_n=\frac{1}{2^{n+1}}\log q_2(2^ns).\]

By Lemma \ref{fitaq}, we have that
\[|x_n|=\frac{1}{2^{n+1}}\left|\log q_1(2^ns) \right|\leq
\frac{1}{2^{n+1}}\frac{16}{2^{2^n\sigma}-2}=o\left(2^{-2^n}\right).\]
Analogously, \[y_n=o\left(2^{-2^n}\right).\]

Thus,
\begin{eqnarray*}
p_i(2s)&=&\exp\left\{-\sum_{k=1}^{n}x_k-\sum_{k=n+1}^{\infty}o\left(2^{-2^k}\right)\right\}\\
&=&\exp\left\{-\sum_{k=1}^{n}x_k-o\left(\sum_{k=n+1}^{\infty}2^{-2^k}\right)\right\}\\
&=&\exp\left\{-\sum_{k=1}^{n}x_k+o\left(2^{-2^n}\right)\right\}\\
&=&\exp\left\{-\sum_{k=1}^{n}x_k\right\}\exp\left\{o\left(2^{-2^n}\right)\right\}\\
&=&\exp\left\{-\sum_{k=1}^{n}x_k\right\}\left(1+o\left(2^{-2^n}\right)\right)\\
&=&\exp\left\{-\sum_{k=1}^{n}x_k\right\}+o\left(2^{-2^n}\right),
\end{eqnarray*}
and we are done.
\qed

%\begin{rem} The convergence of this series is much faster than the
%infinite product or the series obtained by expanding it, since
%they have order of convergence $\mathcal{O}(n^{-1})$.
%\end{rem}

Let us see now how can we evaluate the general term $2^{-n-1}\log q_i(2^ns)$ of the series at even positive integers when $\Delta_K>0$.

Recall that given a Dirichlet character $\chi$ mod $d$, the
generalized Bernoulli numbers \cite{1} are given by
\[\sum_{a=1}^d\chi(a)\frac{te^{at}}{e^{dt}-1}=\sum_{n=0}^{\infty}B_{n,\chi}\frac{t^n}{n!}.\]

Moreover, \[L(1-n,\chi)=-\frac{B_{n,\chi}}{n},\] and
using the functional equation of the $L$-function one can evaluate
$L$ at some positive integers, as given in the following Theorem.

\begin{thm}[\cite{1}]\label{L} Let $\chi$ be a nontrivial primitive character modulo $d$, and let $a$ be 0 if $\chi$ is even and 1 if $\chi$ is odd. Then,
if $n\equiv a\ (mod\ 2)$,
\[L(n,\chi)=(-1)^{1+\frac{n-a}{2}}\frac{g(\chi)}{2i^a}\left(\frac{2\pi}{m}\right)^{n}\frac{B_{n,\overline{\chi}}}{n!},\]
where $g(\chi)$ is the Gauss sum of the character.
\end{thm}

Let now be $d=\Delta_K>0$. Then, $\chi$ is an even quadratic character $mod\ d$. Therefore, for each $n\in\mathbb{N}$ even, one has
\begin{equation}\label{Ln}L(n,\chi)=(-1)^{1+\frac{n}{2}}\frac{\sqrt d}{2}\left(\frac{2\pi}{d}\right)^{n}\frac{B_{n,{\chi}}}{n!},\end{equation}
and
\begin{equation}\label{zn}\zeta(n)=(-1)^{1+\frac{n}{2}}\frac{(2\pi)^n}{2}\frac{B_{n}}{n!}.\end{equation}
From these equalities, we deduce the following.

\begin{prop}\label{calc} Assume that $d=\Delta_K>0$. Then, for each even natural number $n\geq2$, we have
\begin{equation}\label{q1}q_1(n)=\frac{2d^n}{{2n\choose n}\sqrt d}\frac{B_{2n}}{B_{n,\chi}B_{n}}\prod_{p|d}(1+p^{-n}),\end{equation}
\begin{equation}\label{q2}q_2(n)=\frac{\sqrt d}{d^n}\frac{B_{n,{\chi}}}{B_n}\prod_{p|d}(1-p^{-n})^{-1}.\end{equation}
\end{prop}

\begin{proof} It follows immediately from \eqref{Ln}, \eqref{zn}, and the definition of $q_1$ and $q_2$ \eqref{qs}.
\end{proof}

Hence, by using Proposition \ref{calc} and Theorem \ref{prop2} we obtain series of very fast convergence to evaluate $p_1$ and $p_2$ at
even positive integers.

To see an example, let $\chi$ be the primitive
character modulo 5, and let us evaluate $p_1(2)$.
One the one hand, Taking the first 10 terms of the infinite
product one obtains 2 correct digits. On the other hand, taking also the first 10 terms in our series one obtains 619 correct digits. The following table shows the aproximate error when taking $n$ terms of
our series.

\vspace{3mm}

\begin{center}
\begin{tabular}{| c | c |}
    \hline
    \textbf{N}   & $p_1(2)-\exp\left\{-\sum_{k=1}^{N}\frac{1}{2^k}\log q_1(2^k)\right\}$    \\[3pt]
    \hline
    1    &   $10^{-2}$       \\
    2    &   $10^{-3}$       \\
    3    &   $10^{-6}$       \\
    4    &   $\ 10^{-11}$       \\
    5    &   $\ 10^{-21}$       \\
    6    &   $\ 10^{-41}$       \\
    7    &   $\ 10^{-79}$       \\
    8    &   $\ \ 10^{-157}$       \\
    9    &   $\ \ 10^{-311}$       \\
    10    &   $\ \ 10^{-620}$       \\
    11    &   $\ \ \ 10^{-1237}$       \\
    12    &   $\ \ \ 10^{-2470}$       \\
    \hline
\end{tabular}
\end{center}

\vspace{4mm}

\section*{Acknowledgements}

The author thanks Joan-C. Lario for all his comments and suggestions.
%The author was supported by grants MTM2011-27739-C04-01 (Spain) and 2009SGR345 (Catalunya).

\vspace{4mm}

\end{document}